\documentclass[12pt]{amsart}
\usepackage[utf8]{inputenc}
\usepackage[english]{babel}
\usepackage{amsmath,amsfonts,amssymb,amsthm,amsxtra}
\usepackage[mathscr]{eucal}
\theoremstyle{definition}
\newtheorem {definition}{Definition}[section]

\theoremstyle{plain}
\newtheorem {theorem}{Theorem}[section]
\newtheorem {proposition}{Proposition}[section]
\newtheorem {corollary}{Corollary}[section]
\newtheorem {lemma}{Lemma}[section]

\newtheorem {example}{Example}[section]

\newtheorem*{Theorem}{Theorem}
\newtheorem*{Lemma}{Lemma}
\newtheorem*{Definition}{Definition}
\newtheorem {remark}{Remark}[section]
\newcommand{\dist}{\operatornamewithlimits{dist}}

\def\ar{a\kern-.370em\raise.16ex\hbox{\char95\kern-0.53ex\char'47}\kern.05em}
\def\ees{{\accent"5E e}\kern-.385em\raise.2ex\hbox{\char'23}\kern-.08em}
\def\EES{{\accent"5E E}\kern-.5em\raise.8ex\hbox{\char'23 }}
\def\ow{o\kern-.42em\raise.82ex\hbox{
		\vrule width .12em height .0ex depth .075ex \kern-0.16em \char'56}\kern-.07em}
\def\OW{O\kern-.460em\raise1.36ex\hbox{
		\vrule width .13em height .0ex depth .075ex \kern-0.16em \char'56}\kern-.07em}

\title{Fedoryuk values and stability of global H\"{o}lderian error bounds for polynomial functions}

\author[HUY-VUI H\`A]{HUY-VUI H\`A$^{\dag}$}
\address{$^{\dag}$Thang Long Institute of Mathematics and Applied Sciences,
     \newline \indent  Nghiem Xuan Yem Road,
     \newline \indent 
      Hoang Mai, District, Hanoi, Vietnam}
\email{hhvui@math.ac.vn} 

\author[PHI-D\~{U}NG HO\`{A}NG]{PHI-D\~{U}NG HO\`{A}NG$^{\ddag}$}
\address{$^{\ddag}$Department of Mathematics - Faculty of Fundamental Sciences,
	\newline \indent Laboratory of Applied Mathematics and Computing,
	\newline \indent  Posts and Telecommunications Institute of Technology,
	\newline \indent 
	Km10 Nguyen Trai Road, Ha Dong District, Hanoi, Vietnam}
\email{dunghp@ptit.edu.vn}



\begin{document}

\maketitle

\begin{abstract}
	In this paper we study the stability of a global H\"{o}lderian error bound of the sublevel set $[f \le t]$ under perturbation of $t$, where $f$ is a polynomial function in $n$ real variables. Firstly, we give two formulas which compute the set
	$$H(f) := \{ t \in \mathbb{R}: [f \le t]\ \text{has a global H\"{o}lderian error bound}\}$$ via some special Fedoryuk values of $f$. Then, based on these formulas, we can determine the stability type of a global H\"{o}lderian error bound of $[f \le t]$ for any value $t \in \mathbb{R}$. 
\end{abstract}
\pagestyle{plain}

\section{Introduction}
Let $f : \mathbb{R}^n \to \mathbb{R}$ be a polynomial function. For $t \in \mathbb{R}$, put $$[f \le t]:=\{x \in \mathbb{R}^n | f(x) \le t\}$$ and $[a]_+ := \max\{0, a\}$. 
 \begin{definition}\cite{Ha}
	We say that the nonempty set $[f \le t]$ has a global H\"{o}lderian error bound (GHEB for short) if there exist $\alpha, \beta, c >0$ such that 
	\begin{equation}\label{Eqn1}
	[f(x) - t]_+^\alpha + [f(x) - t]_+^\beta \ge c\dist(x, [ f \le t])\ \text{for all}\ x \in \mathbb{R}^n.
	\end{equation}
\end{definition}
Note that, if $\alpha = \beta = 1$, then (\ref{Eqn1}) becomes a global Lipschitzian error bound for $[f \le t]$.

The existence of error bounds have many important applications, including sensitivity analysis, convergence analysis in optimization problems, variational inequalities...
After the earliest work by Hoffman (\cite{Hoff}) and extended paper of Robinson (\cite{Ro}), the study of error bounds has received rising awareness in many papers of mathematical programming in recent years, see \cite{LL, WP, LS, Y, LiG1, LiG2, Ha, Ng, LMP, DHP} (for the case of polynomial functions) and \cite{Hoff, Ro, M, AC, LiW, K, KL, P, LP, Luo, Jo, NZ, CM, LTW, I, BNPS, DL} (for non-polynomial cases). The reader is referred to survey papers \cite{LP, P, Az, I} and the references therein for the theory and applications of error bounds. 

Studying the stability of error bounds under perturbation is fundamental and hard problem. It has been investigated recently in the works of Daniel, Luo-Tseng, Deng, Ngai-Kruger-Th\'{e}ra, Kruger-Ngai-Th\'{e}ra, Kruger-L\'{o}pez-Th\'{e}ra,... (see \cite{Da, LT, D, NKT, KNT, KLT}).

In this paper, we study stability of a global H\"{o}lderian error bound for the set $[f \le t]$ under a perturbation of $t$, i.e. the perturbation of $f$ by a constant term. The following questions arise
\begin{enumerate}
	\item[1.] Suppose that $[f \le t]$ has a GHEB, when does there exist an open interval $I(t) \subset \mathbb{R}, t \in I(t)$, such that for any $t' \in I(t)$, $[f \le t']$ has also a GHEB?
	\item[2.] Suppose that $[f \le t]$ does not have GHEB, when does there exist an open interval $I(t) \subset \mathbb{R}, t \in I(t)$, such that for any $t' \in I(t)$, $[f \le t']$ also does not have GHEB?
	\item[3.] Are there other types of stability which are different from types in questions 1 and 2?
\end{enumerate}
To classify the stability types of GHEB, our idea is computing the set $$H(f) := \{ t \in \mathbb{R}: [f \le t]\ \text{has a global H\"{o}lderian error bound}\}.$$ It turns out that the set $H(f)$ can be determined via some speacial values of the Fedoryuk set of $f$. 

According \cite{KOS}, the Fedoryuk set $F(f)$ of a polynomial $f$ is defined by $$ F(f):=\{t \in \mathbb{R} : \exists \{x^k\} \subset \mathbb{R}^n, \|x^k\| \to \infty, \|\nabla f(x^k)\| \to 0, f(x^k) \to t\}. $$
We will show that there exists a value $h(f) \in F(f) \cup \{\pm \infty\}$, which will be called the {\em threshold} of global H\"{o}lderian error bounds of $f$ and a subset $F^1(f)$ of $F(f)$, such that $$\text{Either}\ H(f) = [h(f), +\infty) \setminus F^1(f)\ \text{or}\ H(f) = (h(f), +\infty) \setminus F^1(f). $$
Since $F^1(f)$ is a semialgebraic subset of $\mathbb{R}$, this formula allows us answer the questions 1 and 2. Moreover, we can discover some other types of stability which are different from the types in questions 1-2 and give the list of all possible types of stability. 

The paper is organized as follows. In Section 2, we give two different formulas for computing the set $H(f)$. The first formula is based on criterion for the existence of GHEB for $[f \le t]$, given in \cite{Ha}. The second formula follows from a new criterion for the existence of global H\"{o}lderian error bounds. In Section 3, the relationship between $H(f)$ and the set of Fedoryuk values of $f$ will be established. In Section 4, we use the formulas of $H(f)$ and relationship between $H(f)$ and $F(f)$ to study our problems. It turns out that $F(f)$ is a semialgebraic subset of $\mathbb{R}$, hence $F(f)$ is either empty, or a finite set or a disjoint of finite number of points and intervals. Therefore, it is convenient to consider each of these cases separately.

In Subsection 4.1, we consider the case $F(f) = \emptyset$. In this case, $H(f) = (\inf f, +\infty)$ or $H(f) = [\inf f, +\infty)$ (Theorem \ref{thm41}). Therefore, there are two stability types of GHEB if $H(f) = [\inf f, +\infty)$. Namely, any point $t$ of $(\inf f, +\infty)$ is {\em y-stable}, by this we mean that $t \in H(f)$ and there exists an open interval $I(t)$ such that $t \in I(t) \subset H(f)$. Besides, $t = \inf f$ is {\em y-right stable}, by this we mean that $t \in H(f)$ and there exists $\epsilon > 0$ such that $[t, t + \epsilon) \subset H(f)$ and $(t - \epsilon, t ) \cap H(f) = \emptyset$. Note that, for almost every polynomial $f$, $F(f) = \emptyset$. Hence, $H(f) = (\inf f, +\infty)$ or $[\inf f, +\infty)$ if $f$ is {\em generic} (Remark \ref{remark41}).
	
In Subsection 4.2, we consider the case when $F(f)$ is a non-empty finite set. In this case, we show that
	\begin{itemize}
		\item $H(f) \ne \emptyset$ (Proposition \ref{Prop3.1});
		\item Beside of y-stable type and y-right stable, there are at most 4 other stability types of GHEB. We have
		\begin{description}
			\item[Case A] If $h(f) = -\infty$, then there are 2 types
			\begin{enumerate}
				\item[(i)] $t$ is y-stable.
				\item[(ii)] $t$ is a {\em n-isolated point}: $t \in \mathbb{R}\setminus H(f)$ and for $\epsilon > 0$ sufficiently small, $(t - \epsilon, t) \cup (t, t + \epsilon) \subset H(f)$. 
			\end{enumerate}
			
			\item[Case B] If $h(f)$ is a finite value, then there are 5 types for all $t \in [\inf f, +\infty)$ 
			\begin{enumerate} 
				\item[1.] $t$ is y-stable; 
				
				\item[2.] $t$ is y-right stable; 			
				
				\item[1'.] $t$ is {\em n-stable}: $t \in [\inf f, +\infty)\setminus H(f)$ and there exists an open interval $I(t)$ such that $t \in I(t) \subset [\inf f, +\infty) \setminus H(f)$;
				
				\item[2'.] $t$ is {\em n-right stable}: $t \in [\inf f, +\infty) \setminus H(f)$ and there exists $\epsilon > 0$ such that $[t, t + \epsilon) \subset [\inf f, +\infty)\setminus H(f)$ and $(t - \epsilon, t ) \cap ([\inf f, +\infty) \setminus H(f)) = \emptyset$; 
				
				\item[3'.] $t$ is {\em n-left stable}: $t \in [\inf f, +\infty) \setminus H(f)$ and there exists $\epsilon > 0$ such that $(t - \epsilon, t] \subset [\inf f, +\infty) \setminus H(f)$ and $(t, t + \epsilon) \cap H(f) \ne \emptyset$;
				
				\item[4'.] $t$ is a n-isolated point;
			\end{enumerate}
		Note that:
		\begin{itemize}
			\item If $t$ is y-right stable or $t$ is n-left stable, then it is necessarily that $t = h(f)$;
			\item If $t$ is n-right stable, then it is necessarily that $t = \inf f < h(f)$ and $f^{-1}(\inf f) \ne \emptyset$.
		\end{itemize}
		  			
		\end{description}
	\item We can determine the type of stability of any $t \in [\inf f, +\infty)$ (Theorem \ref{thm43});
	\item We give an estimation of the number of connected components of $H(f)$ (Theorem \ref{thm44});
	\end{itemize}	

	In Subsection 4.3, we consider the case when $\# F(f) = +\infty$. In this case 
	\begin{itemize}
		\item Any value $t$ of $[\inf f, +\infty)$ belongs to one of the following types
		\begin{enumerate} 
			\item[1.] $t$ is y-stable;
			
			\item[2.] $t$ is y-right stable;
			
			\item[3.] $t$ is {\em y-left stable}: $t \in H(f)$ and there exists $\epsilon > 0$ such that $(t - \epsilon, t] \subset H(f)$ and $(t, t + \epsilon) \cap H(f) = \emptyset$;
			
			\item[4.] $t$ is an {\em y-isolated point}: $t \in H(f)$ and for $\epsilon > 0$ sufficiently small, $(t - \epsilon, t) \cup (t, t + \epsilon) \subset (\inf f, +\infty) \setminus H(f)$;  
			
			\item[1'.] $t$ is n-stable;
			
			\item[2'.] $t$ is n-right stable;
			
			\item[3'.] $t$ is n-left stable;
			
			\item[4'.] $t$ is an n-isolated point.  
		\end{enumerate}
	    \item We can determine the type of stability of any $t \in [\inf f, +\infty)$ (Theorem \ref{thm45}).
	\end{itemize}
We conclude with some examples which illustrates some types of stability.

\section{The set $H(f)$}
\subsection{The first formula of $H(f)$}\quad\\
Let $f: \mathbb{R}^n \to \mathbb{R}$ be a polynomial function and $t \in \mathbb{R}$.
\begin{definition}[\cite{DHN, Ha}] We say that
\begin{enumerate}
\item[(i)] A sequence $\{x^k\} \subset \mathbb{R}^n$ is the first type of $[f \le t]$ if 
\begin{align*}
\|x^k\|&\to \infty,\\ 
f(x^k) > t, f(x^k)&\to t,\\
\exists \delta > 0 \ \text{s.t.}\ \dist(x^k, [f \le t])& \ge \delta.
\end{align*} 
\item[(ii)] A sequence $\{x^k\} \subset \mathbb{R}^n$ is the second type of $[f \le t]$ if 
\begin{align*}
\|x^k\|&\to \infty,\\ 
\exists M \in \mathbb{R}: t < f(x^k)& \le M < +\infty,\\
\dist(x^k, [f \le t])& \to +\infty.
\end{align*}
\end{enumerate}
\end{definition}
\begin{theorem}[\cite{Ha}]\label{Thm2.1}
The following statements are equivalent:
\begin{enumerate}
\item[(i)] There are no sequences of the first or second types of $[f \le t]$.
\item[(ii)] $[f \le t]$ has a GHEB, i.e. there exist $\alpha, \beta, c > 0$ such that 
$$[f(x) - t]_+^\alpha + [f(x) - t]_+^\beta \ge c \dist(x, [ f \le t])\ \text{for all}\ x \in \mathbb{R}^n.$$
\end{enumerate}
\end{theorem}
 

Put \begin{align*}
F^1(f) &= \{t \in \mathbb{R}:\exists \{x^k\} \subset \mathbb{R}^n,\{x^k\}\ \text{is a sequence of the first type of}\ [f \le t]\},\\
F^2(f) &= \{t \in \mathbb{R}:\exists \{x^k\} \subset \mathbb{R}^n, \{x^k\}\ \text{is a sequence of the second type of}\ [f \le t]\}.
\end{align*}
\begin{definition}
	Put $$h(f) = \begin{cases}
	\inf f\ \text{ if }\ F^2(f) = \{\inf f\}\ \text{or}\ F^2(f) = \emptyset,\\
	+\infty\ \text{ if }\ F^2(f) = \mathbb{R},\\
	\sup\{t \in \mathbb{R}: t \in F^2(f) \}\ \text{ if }\ F^2(f) \ne \emptyset\ \text{and}\ F^2(f) \ne \mathbb{R}.  
	\end{cases}$$
	We call $h(f)$ the {\em threshold} of global H\"{o}lderian error bounds of $f$.
\end{definition}
\begin{theorem}[The first formula of $H(f)$]\label{Main}
	We have 
	\begin{enumerate}
		\item[(i)] If $h(f) = \inf f$, then $H(f) = [\inf f, +\infty) \setminus F^1(f)$;
		\item[(ii)] If $h(f) = +\infty$, then $H(f) = \emptyset$;
		\item[(iii)] If $h(f) \in F^2(f)$, then $H(f) = (h(f), +\infty) \setminus F^1(f)$;
		\item[(iv)] If $h(f) \notin F^2(f)$, then $H(f) = [h(f), +\infty) \setminus F^1(f)$. 
	\end{enumerate}
\end{theorem}
\begin{proof}
	Clearly, if $t \in F^2(f)$ and $\inf f \le t' \le t$, then $t' \in F^2(f)$. Hence,
	\begin{align*}
	\text{either}\ F^2(f) &= \emptyset,\\
	\text{or}\ F^2(f) &= \mathbb{R},\\
	\text{or}\ F^2(f) &= (\inf f, h(f)]\ \text{if}\ h(f) \in F^2(f),\\
	\text{or}\ F^2(f) &= (\inf f, h(f))\ \text{if}\ h(f) \notin F^2(f).
	\end{align*}
	Therefore, Theorem \ref{Main} follows from Theorem \ref{Thm2.1}.
\end{proof}
\subsection{A new criterion of the existence of a GHEB of $[f \le t]$ and the second formula of $H(f)$}\quad\\
Let $d$ be the degree of a polynomial $f$. By a linear change of coordinates, we can put $f$ in the form $$ f(x_1, \dots, x_n) = a_0x_n^d + a_1(x_1, \dots, x_{n-1})x_n^{d-1} + \dots + a_d(x_1, \dots, x_{n-1})\ (*), $$ where $a_0 \ne 0$ and $a_i(x_1, \dots, x_{n-1})$ are polynomials in $(x_1, \dots, x_{n-1})$, where degrees $\deg a_i\le i, i =1, \dots, d$.

Put $V_1 = \{x \in \mathbb{R}: \dfrac{\partial f}{\partial x_n}(x) = 0\}$. 
\begin{definition}\label{Def2.3}
	We say that 
	\begin{enumerate}
		\item[(i)] A sequence $\{x^k\}$ is of the first type of $[f \le t]$ w.r.t $V_1$ if \begin{align*}
		\|x^k\|& \to \infty,\\ 
		f(x^k) > t,& f(x^k) \to t,\\
		\dist(x^k, [f \le t]) &\ge \delta > 0,\\
		\text{and}\ \{x^k\} &\subset V_1.
		\end{align*} 
		\item[(ii)] A sequence $\{x^k\}$ is of the second type of $[f \le t]$ w.r.t $V_1$ if \begin{align*}
		\|x^k\|& \to \infty,\\ 
		t< f(x^k)&\le M < +\infty,\\
		\dist(x^k, [f \le t]) &\to \infty,\\
		\text{and}\ \{x^k\} &\subset V_1.
		\end{align*}
	\end{enumerate}
\end{definition}
Let $f$ be of the form $(*)$. Put
\begin{align*}
P^1(f) &= \{t \in \mathbb{R}: [f \le t]\ \text{has a sequence of the first type w.r.t. $V_1$}\};\\
P^2(f) &= \{t \in \mathbb{R}: [f \le t]\ \text{has a sequence of the second type w.r.t. $V_1$}\};\\
P(f) &= \{t \in \mathbb{R}: \exists \{x^k\} \subset\mathbb{R}^n, \|x^k\|\to\infty, \frac{\partial f}{\partial x_n}(x^k) = 0, f(x^k) \to t \}.
\end{align*}

\begin{theorem}\label{Thm2.5}
	Let $f$ be of the form $(*)$. Then the following statements are equivalent
	\begin{enumerate}
		\item[(i)] There are no sequences of the first or second types of $[f \le t]$ w.r.t $V_1$;
		\item[(ii)] $\exists \alpha_1, \beta_1, c > 0$ such that $$ [f(x) - t]_+^{\alpha_1} + [f(x) - t]_+^{\beta_1} \ge c_1 \dist(x, [ f \le t]), $$ for all $x \in V_1$;
		\item[(iii)] $\exists \alpha_1, \beta_1, c > 0$ such that $$ [f(x) - t]_+^{\alpha_1} + [f(x) - t]_+^{\beta_1} + [f(x) - t]_+^{\frac{1}{d}} \ge c \dist(x, [ f \le t]), $$ for all $x \in \mathbb{R}^n$;
		\item[(iv)] $[f \le t]$ has a global H\"{o}lderian error bound.
	\end{enumerate}
\end{theorem}
\begin{proof}\quad\\
	We will prove that $(i) \Rightarrow (ii) \Rightarrow (iii) \Rightarrow (iv) \Rightarrow (i)$.\\
	Proof of $(i) \Rightarrow (ii):$\\
	For $\tau > 0$, put 
	$$\psi(\tau) := \begin{cases}
		0 &\text{if}\ [f(x) - t]_+=\tau\ \text{is empty}\\
		\sup\limits_{[f(x) - t]_+ = \tau, x \in V_1} \dist(x, [f \le t]) &\text{if}\ [f(x) - t]_+=\tau\ \text{is not empty}
	\end{cases}.$$
	By (i), $\psi(\tau)$ is well defined on $[0, +\infty)$. Moreover, it follows from Tarski-Seidenberg theorem (see, for example, \cite{BCR, C, HP}), $\psi(\tau)$ is a semialgebraic function.
	
	To prove (ii), it is important to know the behavior of $\psi(\tau)$, as $\tau  \to 0$ or $\tau \to +\infty$. We distinguish 4 possibilities
	\begin{enumerate}
		\item[(a)] $\psi(\tau) \equiv 0$ for $\tau$ sufficiently small and $\psi(\tau) \equiv 0$ for $\tau$ sufficiently large;
		\item[(b)] $\psi(\tau) \equiv 0$ for $\tau$ sufficiently small and $\psi(\tau) \not\equiv 0$ for $\tau$ sufficiently large;
		\item[(c)] $\psi(\tau) \not\equiv 0$ for $\tau$ sufficiently small and $\psi(\tau) \equiv 0$ for $\tau$ sufficiently large;
		\item[(d)] $\psi(\tau) \not\equiv 0$ both for $\tau$ sufficiently small and $\tau$ sufficiently large.
	\end{enumerate}
	We will prove (i) $\Rightarrow$ (ii) for the case (d) because the proofs of other cases are similar.
	 
	In this case, since $\psi(\tau)$ is semialgebraic and $\psi(\tau) \not\equiv 0$ for any $\tau \in [0, +\infty)$, we have 
	\begin{equation}\label{Eq1}
	\psi(\tau) = a_0\tau^{\tilde{\alpha}} + o(\tau^{\tilde{\alpha}})\ \text{as}\ \tau \to 0,\ \text{where}\ a_0 > 0.
	\end{equation}
	and \begin{equation}\label{Eq2}
	\psi(\tau) = b_0\tau^{\tilde{\beta}} + o(\tau^{\tilde{\beta}})\ \text{as}\ \tau \to +\infty,\ \text{where}\ b_0 > 0.
	\end{equation}
	Clearly, $\tilde{\alpha} > 0$. It follows from (\ref{Eq1}) that there exists $\delta > 0$ such that 
	\begin{equation}\label{Eq3}
	[f(x) - t]_+^{\frac{1}{\tilde{\alpha}}} \ge \frac{a_0}{2}\dist(x, [f \le t]),
	\end{equation}
	for $x \in \{x \in V_1: [f(x) - t]_+ \le \delta\}$.
	
	It follows from (\ref{Eq2}) that there exists $\Delta > 0$ sufficiently large, such that for any $x \in \{x \in V_1: [f(x) - t]_+ \ge \Delta\}$. We have
	\begin{equation}\label{Eq5}
	[f(x) - t]_+ \ge \frac{b_0}{2}\dist(x, [f \le t])
	\end{equation}
	if $\tilde{\beta} \le 0$ and 
	\begin{equation}\label{Eq4} 
	[f(x) - t]_+^{\frac{1}{\tilde{\beta}}} \ge \frac{b_0}{2}\dist(x, [f \le t]),
	\end{equation}
	if $\tilde{\beta} >0$.
		
	Since, by (i), there are no sequences of the second type, the function $\dist(x, [f \le t])$ is bounded on the set
	$$ \{x \in V_1: \delta \le [f(x) - t]_+ \le \Delta\}. $$
	
	This fact, together with (\ref{Eq3}), (\ref{Eq5}) and (\ref{Eq4}), give the proof of (i) $\Rightarrow$ (ii).
	
	\noindent Proof of (ii) $\Rightarrow$ (iii):\\
	The proof is based on the following classical result
	\begin{Lemma}[van der Corput, \cite{G}]
		Let $u(\tau)$ be a real valued $C^d$-function, $d \in \mathbb{N}$, that satisfies $|u^{(d)}(\tau)| \ge 1$ for all $\tau \in \mathbb{R}$. Then the following estimate is valid for all $\epsilon > 0$:
		$$ mes\{\tau \in \mathbb{R}: |u(\tau)| \le \epsilon\} \le (2e)( (d+1)! )^{1/d}\epsilon^{1/d}. $$
	\end{Lemma}
	Suppose that we have (ii). Then 
	\begin{itemize}
		\item If $x \in [f \le t]$, then $\dist(x, [f \le t]) = 0$ and (iii) holds automatically.
		\item If $x \in V_1$, then (iii) follows from (ii).
	\end{itemize}
    Assume that $x \notin [f \le t] \cup V_1$. 	
	
	Clearly
	\begin{itemize}
		\item (ii) holds if and only if there exists $c > 0$ such that 
		\begin{equation}\label{Eq7}
		[f(x) - t]_+ \ge c\min\{\dist(x, [f \le t])^{\frac{1}{\alpha_1}}, \dist(x, [f \le t])^{\frac{1}{\beta_1}}\}
		\end{equation} for all $x \in V_1$.
		\item (iii) holds if and only if there exists $c > 0$  
		\begin{equation}\label{Eq8}
		[f(x) - t]_+ \ge c\min\{\dist(x, [f \le t])^{\frac{1}{\alpha_1}}, \dist(x, [f \le t])^{\frac{1}{\beta_1}}, \dist(x, [f \le t])^{{d}}\}.
		\end{equation} for all $x \in \mathbb{R}^n$.
	\end{itemize}	
	Let $x = (x', x_n) \in \mathbb{R}^{n-1}\times \mathbb{R}, x' = (x_1, \dots, x_{n-1})$. We put $$ u_{x'}(\tau) = \frac{f(x', \tau) - t}{a_0 d!}, \tau \in \mathbb{R} $$ and $$ \Sigma(x') = \{\tau \in \mathbb{R}: |u_{x'}(\tau)| \le \frac{f(x) - t}{|a_0|d!} \}. $$
	Since $u_{x'}^{(d)}(\tau) = 1$, it follows from the van der Corput Lemma that there exists a constant $c > 0$, independent of $x$ such that 
	\begin{equation}\label{Eq9} 
	mes\Sigma(x') \le c(f(x) - t)^{1/d}.
	\end{equation}
	Clearly, $\Sigma(x') \ne \emptyset$ and $\Sigma(x') \ne \mathbb{R}$. Since $\Sigma(x')$ is a closed semi-algebraic subset of $\mathbb{R}$, we have $$ \Sigma(x') = \cup_{i=1}^m [a_i, b_i] \bigcup \cup_{j=1}^s \{c_j\}, $$ where $a_i, b_i, c_j \in \mathbb{R}, i=1, \dots, m; j = 1, \dots,s$, and $$ |u(a_i)| = |u(b_i)| = |u(c_j)| = \dfrac{f(x) - t}{|a_0|d!} .$$ 
	Firstly, we see that $x_n \ne c_j, \forall j = 1, \dots, s$. In fact, since $c_j$ is an isolated point of $\Sigma(x')$, $c_j$ is a local extremum of $u_{x'}(\tau)$. Hence, 	
	$$\frac{d u_{x'}}{d\tau}(c_j) = 0$$ or $\dfrac{\partial f}{\partial x_n}(x', c_j) = 0$ i.e. $(x', c_j) \in V_1$, while by assumption, $x = (x', x_n) \notin V_1$. Thus, $x_n \in \{a_i, b_i; i = 1, \dots, m\}$. 
	
	Without loss of generality, we may assume that $x_n = a_1$. Since $|u_{x'}(a_1)| = |u_{x'}(b_1)|$, we distinguish two cases
	\begin{itemize}
		\item If $u_{x'}(a_1) = -u_{x'}(b_1)$, then there exists $\tau_1 \in [a_1, b_1]$ such that $u_{x'}(\tau_1) = 0$, which means that $f(x', \tau_1) = t$ or $(x', \tau_1) \in f^{-1}(t) \subset [f \le t]$. Hence $$ \dist(x, [f \le t]) \le \dist(x, (x', \tau_1)) = |x_n - \tau_1| \le |a_1 - \tau_1| \le mes\Sigma(x'). $$ Then, by (\ref{Eq9}), (iii) holds.
		\item If $u_{x'}(a_1) = u_{x'}(b_1)$, then, by Rolle's Theorem, there exists $\tau_2 \in [a_1, b_1]$ such that $$ \dfrac{d u_{x'}}{d\tau}(\tau_2) = 0, $$ which means that $(x', \tau_2) \in V_1$. Applying (\ref{Eq7}), there exists $c_1 > 0$ such that 
		\begin{equation*}
		[f(x', \tau_2) - t]_+ \ge c_1 \min\{\dist((x', \tau_2), [f \le t] )^{1/\alpha_1}, \dist((x', \tau_2), [f \le t] )^{1/\beta_1}\}.
		\end{equation*}
		Moreover, since $\tau_2 \in \Sigma(x')$, we have
		\begin{equation}\label{Eq12}
		\begin{aligned}
		f(x) - t &\ge [f(x', \tau_2) - t]_+ \\
		&\ge c_1 \min\{\dist((x', \tau_2), [f \le t] )^{1/\alpha_1}, \dist((x', \tau_2), [f \le t] )^{1/\beta_1}\}.
		\end{aligned}
		\end{equation}
	Let $P(x', \tau_2)$ be the point of $[f \le t]$ such that $$ \dist((x',\tau_2), [f \le t] ) = \dist((x', \tau_2), P(x', \tau_2)). $$
	We have \begin{align*}
	\dist(x, [f \le t]) &\le \dist(x, P(x', \tau_2))\\
	& \le \dist(x, (x', \tau_2)) + \dist((x', \tau_2), P(x', \tau_2))\\
	& \le 2\max\{\dist(x, (x', \tau_2)), \dist((x', \tau_2), P(x', \tau_2))\}.
	\end{align*}
	Now:	
	\begin{itemize}
		\item If $\max\{\dist(x, (x', \tau_2)), \dist((x', \tau_2), P(x', \tau_2))\} = \dist(x, (x', \tau_2))$, then 
		$$ \dist(x, [f\le t]) \le 2\dist((x', \tau_2),x) \le 2mes\Sigma(x') \le 2c(f(x) - t)^{1/d}. $$
		\item If $\max\{\dist(x, (x', \tau_2)), \dist((x', \tau_2), P(x', \tau_2))\} = \dist((x', \tau_2), P(x', \tau_2))$, then 
		$$ \dist(x, [f \le t]) \le 2\dist((x', \tau_2), P(x', \tau_2)) \le 2\dist((x', \tau_2), [f \le t]). $$ 
		Then (iii) follows from (\ref{Eq12}).
	\end{itemize}
	Hence, the implication (ii) $\Rightarrow$ (iii) is proved.	
	\end{itemize}
	\noindent Proof of (iii) $\Rightarrow$ (iv):\\
	Clearly, if (iii) holds, then there are no sequences of the first or second types of $[f \le t]$. Hence, by Theorem \ref{Thm2.1}, (iv) holds.
	
	The proof of (iv) $\Rightarrow$ (i) is straightforward.
\end{proof}
\begin{proposition}\label{Changevar}
	Let $f: \mathbb{R}^n \to \mathbb{R}$ be a polynomial function and $A : \mathbb{R}^n \to \mathbb{R}^n$ be a linear isomorphism. Then we have $$ H(f \circ A) = H(f). $$  		 
\end{proposition}
		\begin{proof}
			Let $y = Ax$ and put $g = f \circ A$. 
			
			Firstly, we prove that $t_0 \in H(g) \Rightarrow t_0 \in H(f)$.
			
			We have $f(y) = f(A\circ A^{-1}(y)) = g(A^{-1}(y))$.
			This implies that 
			\begin{equation}\label{fact15}
			[f(y) - t_0]_+^\alpha + [f(y) - t_0]_+^\beta = [g(A^{-1}(y)) - t_0]_+^\alpha + [g(A^{-1}(y)) - t_0]_+^\beta.
			\end{equation}

			Since $t_0 \in H(g)$, then there exists $\alpha, \beta, c > 0$ such that \begin{equation}\label{fact16}
			[g(A^{-1}(y)) - t_0]_+^\alpha + [g(A^{-1}(y)) - t_0]_+^\beta \ge c\dist(A^{-1}(y), [g \le t_0]).
			\end{equation} 
			Suppose that $\dist(A^{-1}(y), [g \le t_0]) = \|A^{-1}(y) - x_0\|$, where $g(x_0) = t_0$ or $f(A(x_0)) = t_0$. Since $y_0 = Ax_0$ and $A$ is a linear isomorphism, we have $f(y_0) = t_0$ and there exists $c' > 0$ such that 
			$$ c'\|y - y_0\| \ge \|A^{-1}(y) - A^{-1}(y_0)\| \ge \frac{1}{c'}\|y - y_0\|. $$
			It follows that $$ \dist(A^{-1}(y), [g \le t_0]) = \|A^{-1}(y) - A^{-1}(y_0)\| \ge \frac{1}{c'}\|y - y_0\| \ge \frac{1}{c'}\dist(y, [f \le t_0]). $$
			Combining (\ref{fact15}), (\ref{fact16}) and above fact, we have $$ 	[f(y) - t_0]_+^\alpha + [f(y) - t_0]_+^\beta \ge \frac{c}{c'}\dist(y, [f \le t_0]), \forall y \in \mathbb{R}^n,$$ 
			i.e., $t_0 \in H(f)$. The claim $t_0 \in H(f) \Rightarrow t_0 \in H(g)$ is proved similarly.
		\end{proof}
We have the following theorem
\begin{theorem}[The second formula of $H(f)$]\label{second}
	Let $f$ be a polynomial of the form $(*)$. Then we have
	\begin{enumerate}
		\item[(i)] $h(f) = \sup\{t \in \mathbb{R}: t \in P^2(f)\}$;
		\item[(ii)] If $h(f) = \inf f$, then $H(f) = [\inf f, +\infty) \setminus P^1(f)$;
		\item[(iii)] If $h(f) = +\infty$, then $H(f) = \emptyset$;
		\item[(iv)] If $h(f) \in \mathbb{R}$ and $h(f) \in P^2(f)$, then $H(f) = (h(f), +\infty) \setminus P^1(f)$;
		\item[(v)] If $h(f) \in \mathbb{R}$ and $h(f) \notin P^2(f)$, then $H(f) = [h(f), +\infty) \setminus P^1(f)$. 
	\end{enumerate}
\end{theorem}		
\section{The relationship between $H(f)$ and Fedoryuk values}
The relationship between Fedoryuk values and the existence of global H\"{o}lderian error bounds is well-known and has been explored in many previous works, see, for example, \cite{Az, CM, LP, Ha, I}. In this section, we will establish this relationship by proving that $h(f) \in F(f) \cup \{\pm \infty\}$ and $F^1(f) \subset F(f)$. We recall
\begin{definition}
	Let $f : \mathbb{R}^n \to \mathbb{R}$ be a polynomial function. The set {\em of Fedoryuk values} of $f$ is defined by $$ F(f):=\{t \in \mathbb{R} : \exists \{x^k\} \subset \mathbb{R}^n, \|x^k\| \to \infty, \|\nabla f(x^k)\| \to 0, f(x^k) \to t\}. $$
\end{definition}
Moreover, we have 
\begin{lemma}\label{Lemma3.1}
	$F(f)$ is a semialgebraic subset of $\mathbb{R}$.
\end{lemma}
\begin{remark}{\rm 
	It follows from Lemma \ref{Lemma3.1} that either $F(f)$ is empty or $F(f)$ is finite set or $F(f)$ is a union of finitely many points and intervals.
		
	Note that $F(f)$ can be an infinite set, for example (see \cite{Par}), if $f(x,y,z) = x + x^2y + x^4yz$, then $F(f) = \mathbb{R}$ and $F(f^2) = (0, +\infty)$ (see also \cite{KOS} and \cite{Sch}).
}\end{remark}
To prove the lemma, it is more convenient to use the logical formulation of the Tarski-Seidenberg Theorem. Let us to recall it.

A {\em first-order formula} is obtained as follows recursively (see, for example, \cite{BCR, C, HP})
\begin{enumerate}
	\item If $f \in \mathbb{R}[X_1, \dots, X_n]$, then $f=0$ and $f > 0$ are first-order formulas (with free variables $X=(X_1,\dots,X_n)$) and $\{x \in \mathbb{R}^n|f(x) = 0\}$ and $\{x \in \mathbb{R}^n|f(x) > 0\}$ are respectively the	subsets of $\mathbb{R}^n$ such that the formulas $f = 0$ and $f>0$ hold. 
	\item If $\Phi$ and $\Psi$ are first-order formulas, then $\Phi \vee \Psi$ (conjunction), $\Phi \wedge \Psi$ (disjunction) and $\lnot\Phi$ (negation) are also first-order formulas.
	\item If $\Phi$ is a formula and $X$ is a variable ranging over $\mathbb{R}$, then $ \exists X \Phi $ and $\forall X \Phi$ are first-order formulas.
\end{enumerate}
\begin{Theorem}[Logical formulation of the Tarski–Seidenberg Theorem \cite{BCR, C,  HP}]\label{Tarski}
	If $\Phi(X_1, \dots, X_n)$ is a first-order formula, then the set $$\{(x_1, \dots, x_n) \in \mathbb{R}^n :\Phi(x_1, \dots,x_n)\ \text{holds} \}$$ is semialgebraic.
\end{Theorem} 
\begin{proof}[Proof of Lemma \ref{Lemma3.1}]
	We have 
	\begin{align*}
	F(f) = \{t \in \mathbb{R}| \forall \epsilon > 0, \exists \delta >0: \forall & R> 0, \exists x \in\mathbb{R}^n: \|x\|^2 \ge R^2,\\ & \|\nabla f(x)\|^2 \le \delta^2, |f(x) - t| \le \epsilon\}.
	\end{align*}
	It follows from above that the set $F(f)$ can be determined by a first-order formula, hence by the Tarski-Seidenberg Theorem, it is a semialgebraic subset of $\mathbb{R}$.
\end{proof}
The following proposition is contained implicitly in \cite[Proof of Theorem B]{Ha}.
\begin{proposition}\label{Prop2.2}
	$F^1(f) \subset F(f)$.
\end{proposition}
\begin{proof}
	Put $X = \{x \in \mathbb{R}^n : f(x) \ge t\}$. By the metric induced from that of $\mathbb{R}^n$, $X$ is a complete metric space and the function $f: X \to \mathbb{R}$ is bounded from below. Let $t \in F^1(f)$ and $\{x^k\}$ be a sequence of the first type of $[f \le t]$: \begin{align*}
	\|x^k\|&\to \infty,\\ 
	f(x^k)& > t,\\
	f(x^k)&\to t,\\
	\exists \delta > 0 \ \text{s.t.}\ \dist(x^k, [f \le t])& \ge \delta.
	\end{align*} 
	Let $\epsilon_k = f(x^k) - t$. Then $\epsilon_k > 0$ and $\epsilon_k \to 0$ as $k \to +\infty$. Set $\lambda_k = \sqrt{\epsilon_k}$. By the Ekeland's Variational Principle (\cite{E}), there exists a sequence $\{y^k\} \subset X$ such that 
	\begin{align*}
	f(y^k)&\le t + \epsilon_k = f(x^k),\\ 
	\dist(y^k, x^k)&\le \lambda_k
	\end{align*}
	and for any $x \in X, x \ne y^k$, we have \begin{equation}\label{Eq13}
	f(x) \ge f(y^k) - \dfrac{\epsilon_k}{\lambda_k}d(x,y^k), \forall x \in X.
	\end{equation} 
	Since $\dist(y^k, x^k) \le \lambda_k = \sqrt{\epsilon_k} \to 0$ and $\dist(x^k, [f \le t]) \ge \delta > 0$, the ball $B(y^k, \delta/2) = \{x \in \mathbb{R}^n: \dist(y^k, x) \le \delta/2\}$ is contained in $X$. Then, inequality (\ref{Eq13}) implies that $$\dfrac{f(y^k + \tau u) - f(y^k)}{\tau} \ge -\sqrt{\epsilon_k}$$ holds true for every $u \in \mathbb{R}^n, \|u\| = 1$ and $\tau \in [0, \delta/2)$. This gives us $$\langle \nabla f(y^k), u \rangle \ge - \sqrt{\epsilon_k}.$$ Putting $u = -\dfrac{\nabla f(y^k)}{\|\nabla f(y^k)\|}$, we get $\|\nabla f(y^k)\| \le \sqrt{\epsilon_k} \to 0$.  
	
	Clearly $f(y^k) \to t$. Therefore $t \in F(f)$.
\end{proof}
\begin{proposition}\label{Prop2.3}
	If there is a sequence of the second type of $[f \le t]$: \begin{align*}
	\|x^k\|&\to \infty,\\ 
	t < f(x^k)& \le M < +\infty,\\
	\dist(x^k, [f \le t])& \to +\infty.
	\end{align*} then there exists a sequence $\{y^k\}$ of the second type of $[f \le t]$: \begin{align*}
	\|y^k\|&\to \infty,\\ 
	t \le f(y^k)& \le M < +\infty,\\
	\dist(y^k, [f \le t])& \to +\infty.
	\end{align*} with additional conditions \begin{align*} 
	\|\nabla f(y^k)\|&\to 0,\\
	\text{and}\ \lim_{k \to \infty}f(y^k)&\in F(f).
	\end{align*}
	In particular, the segment $[t, M]$ contains at least one point of F(f).
\end{proposition}
\begin{proof}
	Put $X = \{x \in \mathbb{R}^n: f(x) \ge t\}, \epsilon_k = f(x^k) - t$ and $\lambda_k = \dfrac{1}{2}\dist(x^k, [f \le t])$. 
	
	As in the proof of Proposition \ref{Prop2.2}, we can find a sequence $\{y^k\} \subset X$ such that \begin{align*}
	\|y^k\|&\to \infty,\\ 
	t \le f(y^k)& \le t + \epsilon_k = f(x^k) \le M < +\infty,\\
	\lim_{k \to \infty}f(y^k) &\in F(f),\\
	\|\nabla f(y^k)\|&\to 0,\\
	\dist(y^k, x^k)& \le \lambda_k.
	\end{align*}
	Since \begin{align*}
	\dist(y^k, [f \le t]) &\ge \dist(x^k, [f \le t]) - \dist(y^k, x^k)\\
	&\ge \dist(x^k, [f \le t]) - \lambda_k = \dfrac{1}{2}\dist(x^k, [f \le t]),
	\end{align*}
	we have $\dist(y^k, [f \le t]) \to +\infty$. The proposition is proved.
\end{proof}
\begin{proposition}\label{Prop2.4}
	If $h(f) \ne -\infty$ and $\# F(f) < +\infty$, then $h(f) \in F(f)$.
\end{proposition}
\begin{proof}
	Assume that $h(f) \ne -\infty$. By contradiction, suppose that $h(f) \notin F(f)$. Hence, either $F(f) = \emptyset$ or $F(f)$ is a non-empty finite set.
	
	By definition of $h(f)$, $[f \le h(f) - \epsilon]$ has a sequence of second type. Hence, it follows from Proposition \ref{Prop2.3}, $F(f) \ne \emptyset$. Thus, $F(f)$ is a non-empty finite set. Then, for any $\epsilon > 0$ sufficiently small, we have $[h(f) - \epsilon, h(f)] \cap F(f) = \emptyset$ and $h(f) - \epsilon \in F^2(f)$.
	
	Let $\{x^k\}$ be a sequence of the second type of $[f \le h(f) - \epsilon]$:
	$$ h(f)-\epsilon \le f(x^k) \le M, \|x^k\| \to \infty\ \text{and}\ \dist(x^k, [f \le h(f) - \epsilon])\to \infty.$$ By Proposition \ref{Prop2.3}, we may assume that $\|\nabla f(x^k)\| \to 0$ and there exists $t_1 \in F(f) \cap [h(f)-\epsilon, M]$ and $ t_1 = \lim\limits_{k\to\infty}f(x^k)$.  
	
	Let $\delta_1 > 0$ such that $t_1 - \delta_1 \notin F(f)$ and $t_1 - \delta_1 > h(f)$. Since $f(x^k) \to t_1$, we can assume that $f(x^k) > t_1 - \delta_1$ for all $k$. Let $y^k$ be the point of $[f \le t_1 - \delta_1]$ such that $\dist(x^k, [f \le t_1 - \delta_1]) = \|x^k - y^k\|$. Clearly, $y^k \in f^{-1}(t_1 - \delta_1)$.\\
	{\bf Claim:} $\{y^k\}$ is a sequence of second type of $[f \le h(f) - \epsilon]$. 
	\begin{proof}[Proof of Claim]
		Since $t_1 - \delta_1 > h(f), t_1 - \delta_1 \notin F^2(f)$. Hence, for some $A > 0$, we have $\|x^k - y^k\| \le A < +\infty$ for all $k$.
		
		Let $z^k$ be the point of $[f \le h(f) - \epsilon]$ such that $\dist(y^k, [f \le h(f) - \epsilon]) = \|y^k - z^k\|$. We have  
		\begin{align*}
		\dist(y^k, [f \le h(f) - \epsilon]) &\ge \dist(x^k, [f \le h(f) - \epsilon]) - \|x^k - y^k\|\\
		&\ge \dist(x^k, [f \le h(f) - \epsilon]) - A.
		\end{align*} 
		This shows that $\dist(y^k, [f \le h(f) - \epsilon]) \to +\infty$ and the claim is proved.
	\end{proof}
	Since $\{y^k\}$ is a sequence of the second type of $[f \le h(f) - \epsilon]$ and $f(y^k) = t_1 - \delta_1 \notin F(f)$, by Proposition \ref{Prop2.3}, there exists $t_2 \in [h(f)-\epsilon, t_1 - \delta_1] \cap F(f)$. Choose $\delta_2$ such that $t_1 - \delta_2 > h(f)$ and $t_2 - \delta_2 \notin F(f)$. Similarly as in the proof of Claim, we can find a sequence of the second type $\{y'^k\}$ of $[f \le h(f)-\epsilon]$ such that $f(y'^k) = t_2 -\delta_2$ and $t_3 \in F(f)$ such that $h(f)-\epsilon < t_3 < t_2$. 
	
	Making this process iteratively, we see that the interval $[h(f) - \epsilon, M]$ contains a infinite number of points in $F(f)$, which is a contradiction.
\end{proof} 
\section{Types of stability of global H\"{o}lderian error bounds}
We will distinguish 3 cases.
\subsection{Case 1 - $F(f) = \emptyset$}
\begin{theorem}\label{thm41}
	If $F(f) = \emptyset$ then $H(f) = (\inf f, +\infty)$ or $H(f) = [\inf f, +\infty)$.
\end{theorem}
\begin{proof}
	Assume that $F(f) = \emptyset$. Then by Proposition \ref{Prop2.2}, $F^1(f) = \emptyset$. Moreover, it follows from Proposition \ref{Prop2.3} that $F^2(f)$ is also empty.
	
	Hence, by Theorem \ref{Thm2.1}, $H(f) = (\inf f, +\infty)\setminus(F^1(f) \cup F^2(f)) = (\inf f, +\infty)$ or $H(f) = [\inf f, +\infty)\setminus(F^1(f) \cup F^2(f)) = [\inf f, +\infty)$.
\end{proof}
\begin{definition}
	Let $t \in \mathbb{R}$.
	\begin{enumerate}
		\item[1.] $t$ is called {\em y-stable} if $t \in H(f)$ and there exists an open interval $I(t)$ such that $t \in I(t) \subset H(f)$;
		\item[2.] $t$ is called {\em y-right stable} if $t \in H(f)$ and there exists $\epsilon > 0$ such that $[t, t + \epsilon) \subset H(f)$ and $(t - \epsilon, t ) \cap H(f) = \emptyset$.
	\end{enumerate}
	
\end{definition}
\begin{corollary}\label{cor4.1}
	If $F(f) = \emptyset$, then we have two cases
	\begin{enumerate}
		\item[1.] If $H(f) = (\inf f, +\infty)$, then there is only one type of stability of GHEB. Namely, for all $t \in (\inf f, +\infty)$, $t$ is y-stable.
		\item[2.] If $H(f) = [\inf f, +\infty)$, then then there are two stability types of GHEB. Namely, for all $t \in (\inf f, +\infty)$, $t$ is y-stable and for $t = \inf f$, $t$ is y-right stable.
	\end{enumerate}   
\end{corollary}
\begin{remark}\label{remark41}{\rm 
We recall here results of \cite{Ha} about the role that Newton polyhedron plays in studying GHEB's.	

Let $f(x) = \sum a_\alpha x^\alpha$ be a polynomial in $n$ variables. Put $supp(f) = \{\alpha \in (\mathbb{N} \cup \{0\})^n: a_\alpha \ne 0\}$ and denote $\Gamma_f$ the convex hull in $\mathbb{R}^n$ of the set $\{(0,0, \dots, 0)\} \cup supp(f)$. Following \cite{Kou} we call $\Gamma_f$ {\em the Newton polyhedron at infinity of $f$}. 

Let $\Delta$ be a face (of any dimension) of $\Gamma_f$, set: $$ f_\Delta(x):= \sum\limits_{\alpha \in \Delta}a_\alpha x^\alpha.$$
\begin{Definition}[\cite{Kou}]{\rm We say that a polynomial $f$ is nondegenerate with respect to its Newton boundary at infinity (nondegenerate for short), if for every face $\Delta$ of $\Gamma_{f}$ not containing the origin, the system $$ x_i\dfrac{\partial f_\Delta}{\partial x_i} = 0, i = 1, \dots, n.$$ has no solution in $(\mathbb{R}\setminus \{0\})^n$.}
\end{Definition}
\begin{Definition}{\rm A polynomial $f(x)=\sum a_\alpha x^\alpha$ in $n$ variables is said to be convenient if for every $i$, there exists a monomial of $f$ of the form $x_i^{\alpha_i}, \alpha_i > 0$, with a non-zero coefficient.} 
\end{Definition}

\begin{theorem}[\cite{Ha}]\label{thmHa}
	 If $f$ is convenient and nondegenerate w.r.t. its Newton polyhedron at infinity, then there exist $r, \delta > 0$ such that $$\|\nabla f(x)\| \ge \delta\ \text{for}\ \|x\| \ge r \gg 1. $$  In particular, $F(f) = \emptyset$.
\end{theorem}
 Let $\mathbb{R}[x_1, \dots, x_n]$ denote the ring of polynomials in $n$ variables over $\mathbb{R}$.
 
 For $g \in \mathbb{R}[x_1, \dots, x_n]$, as before, $\Gamma_g$ denotes the Newton polyhedron at infinity of $g$. Let $f \in \mathbb{R}[x_1, \dots, x_n]$ be a convenient polynomial.
 
 Put $\Gamma:=\Gamma_f$ and 
 $$ \mathcal{A}_\Gamma = \{g \in \mathbb{R}[x_1, \dots, x_n]: \Gamma_g \subset \Gamma \}. $$
The set $\mathcal{A}_\Gamma$ can be identified to the space $\mathbb{R}^m$, where $m$ is the number of integer points of $\Gamma$. 

Put $\mathcal{B}_\Gamma = \{h \in \mathcal{A}_\Gamma: \Gamma_h = \Gamma\ \text{and $h$ is nondegenerate} \}$. According to \cite{Kou}, $\mathcal{B}_\Gamma$ is an open and dense subset of $\mathcal{A}_\Gamma$. Hence, Theorem \ref{thm41} and \ref{thmHa} show that if $f$ is a {\em generic} polynomial, then $H(f) = (\inf f, +\infty)$ or $H(f) = [\inf f, +\infty)$. By Corollary \ref{cor4.1}, any value $t \in (\inf f, +\infty)$ is y-stable and $t = \inf f$ is y-right stable where $H(f) = [\inf f, +\infty)$
}\end{remark}
\subsection{Case 2 - $F(f)$ is non-empty finite set}
\begin{proposition}\label{Prop3.1}
	If $\# F(f) < +\infty$, then $H(f) \ne \emptyset$.
\end{proposition}
\begin{proof}
	By contradiction, assume that $H(f) = \emptyset$. Since $\# F(f) < +\infty$, we have $\# F^1(f) < +\infty$ (Proposition \ref{Prop2.2}). Then, it follows from the first formula that $H(f) = \emptyset$ if and only if $h(f) = +\infty$ but the later is impossible, since we have
	
	{\bf Claim:} If $h(f) = +\infty$, then $\# F(f) = +\infty$.
	\begin{proof}[Proof of Claim]
		Take $t_1 \in \mathbb{R}$, since $h(f) = +\infty$, $[f \le t_1]$ has a sequence of the second type. By Proposition \ref{Prop2.3}, there exists $M_1 > t_1$ and $a_1 \in [t_1, M_1] \cap F(f)$. Take $t_2$ such that $M_1 < t_2$, then $[f \le t_2]$ has a sequence of the second type. Hence, there exists $M_2 > t_2$ and $a_2$ such that $a_2 \in [t_2, M_2] \cap F(f)$. Continuing this way, we find an infinite sequence $a_1, a_2, a_3, \dots$ of $F(f)$. Therefore, $\# F(f) = +\infty$.
	\end{proof}
	
\end{proof}
Now, we classify the stability types of GHEB in the case when $F(f)$ is a non-empty finite set. 
\begin{definition}\label{def42} Let $t \in \mathbb{R}$.
	\begin{enumerate}
		\item[1.] Recall that $t$ is called y-stable if $t \in H(f)$ and there exists an open interval $I(t)$ such that $t \in I(t) \subset H(f)$;
		\item[2.] Recall that $t$ is called y-right stable if $t \in H(f)$ and there exists $\epsilon > 0$ such that $[t, t + \epsilon) \subset H(f)$ and $(t - \epsilon, t ) \cap H(f) = \emptyset$;
		\item[1'.] $t$ is called {\em n-stable} if $t \in [\inf f, +\infty)\setminus H(f)$ and there exists an open interval $I(t)$ such that $t \in I(t) \subset \mathbb{R} \setminus H(f)$;
		\item[2'.] $t$ is called {\em n-right stable} if $t \in [\inf f, +\infty) \setminus H(f)$ and there exists $\epsilon > 0$ such that $[t, t + \epsilon) \subset [\inf f, +\infty)\setminus H(f)$ and $(t - \epsilon, t ) \cap ([\inf f, +\infty) \setminus H(f)) = \emptyset$;
		\item[3'.] $t$ is called {\em n-left stable} if $t \in [\inf f, +\infty) \setminus H(f)$ and there exists $\epsilon > 0$ such that $(t - \epsilon, t] \subset [\inf f, +\infty) \setminus H(f)$ and $(t, t + \epsilon) \cap H(f) \ne \emptyset$;
		\item[4'.] $t$ is called {\em n-isolated} if $t \in \mathbb{R}\setminus H(f)$ and for $\epsilon > 0$ sufficiently small, $(t - \epsilon, t) \cup (t, t + \epsilon) \subset H(f)$.  
	\end{enumerate}	
\end{definition}
It follows from the first formula that
\begin{theorem}\label{thm43}
	Let $F(f)$ be a non-empty finite set and $t \in [\inf f, +\infty)$. Then, $t$ is one of the following types  
	\begin{description}
		\item[Case A] If $h(f) = -\infty$, then
		\begin{enumerate}
			\item[(i)] $t$ is y-stable if and only if $t \notin F^1(f)$.
			\item[(ii)] $t$ is a n-isolated point if and only if $t \in F^1(f)$. 
		\end{enumerate}
		
		\item[Case B] If $h(f)$ is a finite value, then 
			\begin{enumerate} 
				\item[1.] $t$ is y-stable if and only if $t > h(f)$ and $t\notin F^1(f)$; 
				
				\item[2.] $t$ is y-right stable if and only if $t = h(f)$ and $h(f) \in H(f)$; 			
			
				\item[1'.] $t$ is n-stable if and only if $\inf f < t < h(f)$;
				
				\item[2'.] $t$ is n-right stable if and only if $t = \inf f < h(f)$ and $f^{-1}(\inf f) \ne \emptyset$;
				
				\item[3'.] $t$ is n-left stable if and only if $t = h(f)$ and $h(f) \notin H(f)$;
				
				\item[4'.] $t$ is a n-isolated point if and only if $t > h(f)$ and $t \in F^1(f)$.
			\end{enumerate}

	\end{description}
\end{theorem} 
\begin{remark}{\rm 
		Here, if we have item 2, then we does not have item 3' and vice versa. 	
}\end{remark}
Now, to complete this subsection, we add an estimation of the number of connected components of $H(f)$ for the case $\# F(f) < +\infty$.

Let us denote $C(S)$ the number of connected components of $S \subset \mathbb{R}^n$, we have the following result
\begin{theorem}\label{thm44}
	Let $f: \mathbb{R}^n \to \mathbb{R}$ be an any polynomial of degree $d$. Then, if $\# F(f) < +\infty$, we have $$ C(H(f)) \le (d-1)^{n-1} + 1. $$ 	
\end{theorem}
\begin{proof}
	Put	$$ F_{\mathbb{C}}(f):=\{t \in \mathbb{C} : \exists \{x^k\} \subset \mathbb{C}^n, \|x^k\| \to \infty, \|\nabla f(x^k)\| \to 0, f(x^k) \to t\}. $$
	
	Since $\# F(f) < +\infty $, we have $ \# F_{\mathbb{C}}(f) < +\infty$. Then, according to Theorem 1.1 of \cite{Je}, we have $$\# F(f) \le \# F_{\mathbb{C}}(f) \le (d-1)^{n-1}.$$ Hence, it follows from the first formula that $$C(H(f)) \le (d-1)^{n-1} + 1.$$
\end{proof}
\subsection{Case 3 - $F(f)$ is an infinite set}\quad\\
In this case, the following lemma tells us that the set $H(f)$ has still very simple structure
\begin{lemma}\label{Lemma2.1}
	$H(f)$ is a semialgebraic subset of $\mathbb{R}$.
\end{lemma}
Using the first formula for $H(f)$ (Theorem \ref{Main}), it is enough to show that $F^1(f)$ is semialgebraic. 
\begin{proof}[Proof of Lemma \ref{Lemma2.1}]
	We have 
	\begin{align*}
	F^1(f) = \{t \in \mathbb{R}| \exists \delta > 0,\forall R> 0: \forall &\epsilon >0, \exists x \in\mathbb{R}^n: \|x\|^2 \ge R^2,\\ &0 < f(x) - t < \epsilon, \dist(x, [f \le t]) \ge \delta\},\quad\ \text{(a)}
	\end{align*}
	$$\{x \in \mathbb{R}^n: \dist(x, [f \le t]) \ge \delta \} = \{x \in \mathbb{R}^n: \exists\delta \forall x_0 \in [f \le t], \|x - x_0\|^2 \ge \delta^2\}.\quad \text{(b)}$$
	It follows from (a) and (b) that the set $F^1(f)$ can be determined by a first-order formula, hence it is a semialgebraic subset of $\mathbb{R}$.
\end{proof}
Since $H(f)$ is a semialgebraic subset of $\mathbb{R}$, we have
\begin{corollary}\label{cor2.1}
	If $H(f) \ne \emptyset$ and $H(f) \ne \mathbb{R}$, then it is a union of finitely many points and intervals.
\end{corollary}
By Corollary \ref{cor2.1}, we have to consider three cases
\begin{enumerate}
	\item[(a)] $H(f) = \mathbb{R}$;
	\item[(b)] $H(f) = \emptyset$;
	\item[(c)] $H(f)$ is a non-empty proper semialgebraic subset of $\mathbb{R}$.
\end{enumerate}
\begin{itemize}
	\item In the case (a), we have only one stable type: $t$ is y-stable for all $t \in \mathbb{R}$;
	\item In the case (b), we have only one stable type: $t$ is n-stable for all $t \in \mathbb{R}$;
	\item In the case (c), $H(f)$ is a disjoint union of the sets of the following types: $$I_{(a^1_i, a^2_i)}, I_{[b^1_j, b^2_j)}, I_{(c^1_k, c^2_k]}, I_{[d^1_l, d^2_l]}, A(m), I_{-\infty}, I_{+\infty}. $$ Where
	\begin{enumerate}
		\item $I_{(a^1_i, a^2_i)} = \emptyset$ or $I_{(a^1_i, a^2_i)} = {(a^1_i, a^2_i)}, i = 1, \dots, p$;
		\item $I_{[b^1_j, b^2_j)} = \emptyset$ or $I_{[b^1_j, b^2_j)} = {[b^1_j, b^2_j)}, j = 1, \dots, q$;
		\item $I_{(c^1_k, c^2_k]} = \emptyset$ or $I_{(c^1_k, c^2_k]} = {(c^1_k, c^2_k]}, k = 1, \dots, r$;
		\item $I_{[d^1_l, d^2_l]}=\emptyset$ or $I_{[d^1_l, d^2_l]} = {[d^1_l, d^2_l]}, l = 1, \dots, s$;
		\item $A(m) = \emptyset$ or $A(m)=\{e_1, \dots, e_m\}$, where $e_1, \dots, e_m$ are isolated points;
		\item $I_{-\infty} = \emptyset$ or $I_{-\infty}=(-\infty, a]$ or $I_{-\infty}=(-\infty, a)$, where $a \in \mathbb{R}$;
		\item $I_{+\infty} = \emptyset$ or $I_{+\infty} = [b, +\infty)$ or $I_{+\infty} = (b, +\infty)$, where $b \in \mathbb{R}$. 
	\end{enumerate}	
	Similarly, $\mathbb{R} \setminus H(f)$ is a disjoint union of the sets of the following types: $$I_{(a'^1_i, a'^2_i)}, I_{[b'^1_j, b'^2_j)}, I_{(c'^1_k, c'^2_k]}, I_{[d'^1_l, d'^2_l]}, A'(m'), I'_{-\infty}, I'_{+\infty}. $$   
\end{itemize}
We have the following definition
\begin{definition}\label{def43}
	Let $t \in \mathbb{R}$.
	\begin{enumerate}
		\item[1.] Recall that $t$ is said to be y-stable if $t \in H(f)$ and there exists an open interval $I(t)$ such that $t \in I(t) \subset H(f)$;
		\item[2.] Recall that $t$ is said to be y-right stable if $t \in H(f)$ and there exists $\epsilon > 0$ such that $[t, t + \epsilon) \subset H(f)$ and $(t - \epsilon, t ) \cap H(f) = \emptyset$;
		\item[3.] $t$ is said to be {\em y-left stable} if $t \in H(f)$ and there exists $\epsilon > 0$ such that $(t - \epsilon, t] \subset H(f)$ and $(t, t + \epsilon) \cap H(f) = \emptyset$;
		\item[4.] $t$ is said to be {\em y-isolated} if $t \in H(f)$ and for $\epsilon > 0$ sufficiently small, $(t - \epsilon, t) \cup (t, t + \epsilon) \subset \mathbb{R} \setminus H(f)$; 
		\item[1'.] Recall that $t$ is called n-stable if $t \in [\inf f, +\infty)\setminus H(f)$ and there exists an open interval $I(t)$ such that $t \in I(t) \subset [\inf f, +\infty) \setminus H(f)$;
		\item[2'.] Recall that $t$ is called n-right stable if $t \in [\inf f, +\infty) \setminus H(f)$ and there exists $\epsilon > 0$ such that $[t, t + \epsilon) \subset [\inf f, +\infty)\setminus H(f)$ and $(t - \epsilon, t ) \cap ([\inf f, +\infty) \setminus H(f)) \ne \emptyset$;
		\item[3'.] Recall that $t$ is called n-left stable if $t \in [\inf f, +\infty) \setminus H(f)$ and there exists $\epsilon > 0$ such that $(t - \epsilon, t] \subset [\inf f, +\infty) \setminus H(f)$ and $(t, t + \epsilon) \cap H(f) = \emptyset$;
		\item[4'.] Recall that $t$ is called n-isolated if $t \in [\inf f, +\infty) \setminus H(f)$ and for $\epsilon > 0$ sufficiently small, $(t - \epsilon, t) \cup (t, t + \epsilon) \subset H(f)$.  
	\end{enumerate}
\end{definition}
Using the first formula of $H(f)$, we have
\begin{theorem}\label{thm45}
	Let $H(f)$ be of the form (c) and $t \in [\inf f, +\infty)$. Then we have  
	     \begin{enumerate} 
			\item[1.] $t$ is y-stable if and only if $t$ is an interior point of the sets $$I_{-\infty} \bigcup\cup_{i=1}^p I_{(a^1_i, a^2_i)} \bigcup \cup_{j=1}^q I_{[b^1_j, b^2_j)} \bigcup \cup_{k=1}^r I_{(c^1_k, c^2_k]} \bigcup\cup_{l=1}^s I_{[d^1_l, d^2_l]} \bigcup I_{+\infty}; $$
			
			\item[2.] $t$ is y-right stable if and only if we have $t = b^1_j$ or $t = d^1_l$ or $t = b$ (where $I_{+\infty} = [b, +\infty)$);
			
			\item[3.] $t$ is y-left stable if and only if we have $t = c^2_k$ or $t = d^2_l$ or $t = a$ (where $I_{-\infty} = (-\infty, a]$);
			
			\item[4.] $t$ is an y-isolated point if and only if $t \in A(m)$.  
			
	    	\item[1'.] $t$ is n-stable if and only if $t$ is an interior point of the set: $$I'_{-\infty} \bigcup\cup_{i=1}^{p'} I_{(a'^1_i, a'^2_i)} \bigcup \cup_{j=1}^{q'} I_{[b'^1_j, b'^2_j)} \bigcup \cup_{k=1}^{r'} I_{(c'^1_k, c'^2_k]}\bigcup \cup_{l=1}^{s'} I_{[d'^1_l, d'^2_l]} \bigcup I'_{+\infty}; $$ 
			
			\item[2'.] $t$ is n-right stable if and only if we have $t = b'^1_j$ or $t = d'^1_l$ or $t = b'$ (where $I'_{+\infty} = [b', +\infty)$);
			
			\item[3'.] $t$ is n-left stable if and only if we have $t = c'^2_k$ or $t = d'^2_l$ or $t = a'$ (where $I'_{-\infty} = (-\infty, a']$);
			
			\item[4'.] $t$ is an n-isolated point if and only if $t \in A'(m')$.  
		\end{enumerate}
\end{theorem}
\begin{remark}{\rm 
		In the above list, we collect all types of stability that could theoretically exist. The problem of deciding when this or that type really appears, seems to be very difficult.
}\end{remark}
We finish our paper by considering the following simple example
\begin{example}\label{Ex1}{\rm
		Let $f(x,y) = (y^2-1)^2 + (xy - 1)^2$ (\cite{HT}). Clearly, $f$ is of the form $(*)$.
		
		We have $\dfrac{\partial f}{\partial y} = 4y^3 + 2x^2y - 4y - 2x = 2(2y^3 + x^2y - 2y - x)$. Hence, the roots of $\dfrac{\partial f}{\partial y} = 0$ are: 
		\begin{align*}
		x_1(y) &= \dfrac{1 + \sqrt{-8y^4 + 8y^2 + 1}}{2y}\ \text{and}\ \lim_{y \to 0}x_1(y) = +\infty,\\
		x_2(y) &= \dfrac{1 - \sqrt{-8y^4 + 8y^2 + 1}}{2y}\ \text{and}\ \lim_{y \to 0}x_2(y) = -\infty.
		\end{align*}
		We have
			\begin{align*}
			 \lim_{y \to 0}\dfrac{\partial f}{\partial x}(x_i(y), y) = 0, i =1,2 &\Rightarrow \lim\limits_{(x,y) \in V_1,\|(x,y)\|\to\infty}\|\nabla f(x,y)\| = 0;\\
			\lim\limits_{y \to 0}f(x_i(y), y) = 1, i=1,2 &\Rightarrow \lim\limits_{(x,y) \in V_1,\|(x,y)\|\to\infty}f(x,y) = 1.
			\end{align*}
			Hence $P(f) = \{1\}$.
			
			It is not difficult to show that
			\begin{itemize}
				\item $F^2(f) = [0, 1)$, hence $h(f) = 1$;
				\item and $F^1(f) = \emptyset$.
			\end{itemize}
			  Therefore, by the second formula, $H(f) = [1, +\infty)$. In this example, for any $t \in [0, +\infty)$:
		\begin{itemize}
			\item If $t \in (1, + \infty)$, then $t$ is y-stable;
			\item If $t = 1$, then $t$ is y-right stable;
			\item If $t \in (0,1)$, then $t$ is n-stable;
			\item If $t = 0$, then $t$ is n-right stable.
		\end{itemize}
}\end{example}
\subsection*{Acknowledgments}
This research was partially supported by National Foundation for Science and Technology Development (NAFOSTED), Vietnam; Grant numbers 101.04-2017.12 of the first author and 101.04-2019.302 of the second author.

\bibliographystyle{amsalpha}

\end{document}